\documentclass[12pt]{amsart}
\usepackage{amssymb}
\usepackage[margin=1in]{geometry}

\usepackage{multirow}

\title{Constructions of exotic group $C$*-algebras}
\author{Matthew Wiersma}
\address{Department of Pure Mathematics, University of Waterloo, Waterloo, ON, Canada N2L 3G1}
\email{mwiersma@uwaterloo.ca}


\newtheorem{theorem}{Theorem}[section]

\newtheorem{prop}[theorem]{Proposition}

\newtheorem{proposition}[theorem]{Proposition}

\theoremstyle{remark}
\newtheorem{remark}[theorem]{Remark}

\theoremstyle{definition}
\newtheorem{defn}[theorem]{Definition}
\newtheorem{example}[theorem]{Example}

\newcommand{\fn}{\!:}
\newcommand{\C}{\mathbb C}
\newcommand{\R}{{\mathbb R}}

\newcommand{\A}{\mathcal A}

\newcommand{\Hi}{\mathcal{H}}
\newcommand{\lla}{\left\langle}
\newcommand{\rra}{\right\rangle}
\newcommand{\mc}{\mathcal}

\newcommand{\tn}{\textnormal}

\newcommand{\Z}{\mathbb Z}

\newcommand{\F}{\mathbb F}

\begin{document}

\begin{abstract}
Let $\Gamma$ be a discrete group. When $\Gamma$ is nonamenable, the reduced and full group $C$*-algebras differ and it is generally believed that there should be many intermediate $C$*-algebras, however few examples are known. In this paper we give new constructions and compare existing constructions of intermediate group $C$*-algebras for both generic and specific groups $\Gamma$.
\end{abstract}

\maketitle

\section{Introduction}

Let $G$ be a locally compact group. Then $G$ is amenable if and only if $C^*(G)$ and $C^*_r(G)$, the full and reduced group $C$*-algebras of $G$, coincide. So when $G$ is nonamenable, this begs the question as to whether there are any intermediate or {\it exotic} group $C$*-algebras between $C^*_r(G)$ and $C^*(G)$. It is generally believed that there should be many such exotic group $C$*-algebras, however few examples are known.

An early class of intermediate group $C$*-algebras is due to Bekka, Kaniuth, Lau, and Schlichting. Let $G$ be a locally compact group and $G_d$ be the group $G$ endowed with the discrete topology. In their 1996 paper \cite{bkls}, these authors give a characterization of when $\lambda_G$, the left regular representation of $G$ viewed as a representation of $G_d$, is weakly contained in $\lambda_{G_d}$, the left regular representation of $G_d$. For a large class of groups $G$ where $\lambda_{G_d}$ does not weakly contain $\lambda_G$, the group $C$*-algebra $C^*_{\lambda_G}(G_d)$ lies strictly between the reduced and full group $C$*-algebras.

Another early class of exotic group $C$*-algebras is produced in Bekka's 1999 paper \cite{b}. Here, Bekka demonstrates a class of arithmetic groups $\Gamma$ for which $C^*(\Gamma)$ is not residually finite dimensional. For such $\Gamma$, $C_{\mc F}^*(\Gamma)$ is an exotic group $C$*-algebra when $\Gamma$ is maximally almost periodic.

Recently, Brown and Guentner introduced the notion of ideal completions for discrete groups $\Gamma$ \cite{bg}. This allows one to construct group $C$*-algebras of $\Gamma$ associated to $\ell^p(\Gamma)$ (denoted $C^*_{\ell^p}(\Gamma)$) for $1\leq p<\infty$. It turns out that the only interesting case to consider is when $p\in (2,\infty)$ \cite[Proposition 2.11]{bg}. Let $\F_d$ be a free group on $2\leq d<\infty$ generators. In \cite[Proposition 4.2]{bg}, Brown and Guentner show that there exists a $p\in (2,\infty)$ so that $C^*_{\ell^p}(\Gamma)$ is an intermediate $C$*-algebra. Subsequently, Okayasu was able to adapt arguments due to Haagerup to show that each of these $C$*-algebras are distinct for $2\leq p<\infty$ \cite{o}, thus giving an infinite chain of intermediate $C$*-algebras associated to $\mathbb F_d$. It follows that the $C$*-algebras $C^*_{\ell^p}(\Gamma)$ are all distinct for any discrete group $\Gamma$ containing a copy of the free group.

In this paper, we aim to compare these existing constructions and introduce new constructions of exotic $C$*-algebras associated to a discrete group $\Gamma$. In section 2 we provide the necessary background on ideal completions and prove some supplementary results. Section 3 introduces an intuitive lattice structure which can be placed on the group $C$*-algebras of $\Gamma$. With the exception of examples given in section 2, all of our new constructions of $C$*-algebras arise by using this lattice structure. In sections 4 and 5 we focus our attention towards studying intermediate $C$*-algebras on specific groups. Section 4 studies $SL_n(S)$ where $S$ is a dense subring of $\R$ while section 5 analyzes $SL_n(\Z)$. Specific attention is paid in comparing the exotic group $C$*-algebras associated to $\ell^p$ with the constructions due to \cite{bkls} and \cite{b}, respectively.

\section{Ideal completions}

As mentioned in the introduction, this section aims to provide a quick introduction to ideal completions, a concept introduced in \cite{bg}. Along the way, we generalize a result on amenability to the setting of homogeneous spaces and prove that the induced representation of an $\ell^p$-representation remains an $\ell^p$-representation. This section concludes with a class of examples of exotic group $C$*-algebras for $\F_\infty$, the free group on countably many generators, which lie off the chain of $C$*-algebras associated to $\ell^p$.

Let $\Gamma$ be a discrete group and $D\lhd\ell^\infty(\Gamma)$ an algebraic ideal. A (unitary) representation $\pi\fn \Gamma\to B(\Hi)$ is said to be a $D$-representation if $\Hi$ admits a dense linear subspace $\Hi_0$ so that $\pi_{x,x}\in D$ for every $x\in \Hi_0$. It is easily verified that the $D$-representations are closed under tensor products against arbitary representations of $\Gamma$ and under arbitrary direct sums of $D$-representations \cite[Remarks 2.4, 2.5]{bg}. Associated to these $D$-representations, we define the $C$*-seminorm $\|\cdot\|_D$ on the group ring $\C[\Gamma]$ by
$$\|x\|_D=\sup\{\|\pi(x)\| : \pi\tn{ is a $D$-representation}\} $$
and denote the ``completion'' of $\C[\Gamma]$ with respect to $\|\cdot\|_D$ by $C^*_D(\Gamma)$. This process of producing $C$*-algebras is called ideal completions. For our purposes, the most interesting algebraic ideal to consider is when $D=\ell^p(\Gamma)$. Later, we will also introduce ideals $D_p$ which are defined with respect to a fixed subgroup $H\leq \Gamma$ and generalize the ideals $\ell^p$.

In the construction of these ideal completions, it is desirable for the ideal $D$ to be translation invariant (under both left and right translation). This guarantees that when $D$ is nonzero, the left regular representation is a $D$-representation and, hence, the $D$-representations separate points of $\C[\Gamma]$ (see the remark following Definition 2.6 in \cite{bg}). This also ensures the desirable property that if $\varphi$ is a positive definite function which lies in $D$, then the GNS representation associated to $\varphi$ is a GNS representation \cite[Lemma 3.1]{bg} and, hence, that $\varphi$ extends to a positive linear functional on $C^*_D(\Gamma)$.

Brown and Guentner, recognizing the importance of the case when $D=\ell^p$, developed some basic theory of $\ell^p$ ideal completions. They demonstrated that for every $p\in [1,2]$, the ideal completion $C_{\ell^p}(\Gamma)$ simply gives the reduced $C$*-algebra $C^*_r(\Gamma)$ \cite[Proposition 2.11]{bg} and showed that if there exists $p\in [1,\infty)$ so that $C^*_{\ell^p}(\Gamma)=C^*(\Gamma)$, then $\Gamma$ is amenable \cite[Proposition 2.12]{bg} (in their proof of \cite[Proposition 2.12]{bg}, Brown and Guentner assume that $\Gamma$ is countable. This assumption is not necessary as will be demonstrated in Proposition \ref{1}). Rephrasing this proposition, we get the characterization that $\Gamma$ is amenable if and only if there exists $p\in [1,\infty)$ so that $C^*_{\ell^p}(\Gamma)=C^*(\Gamma)$ if and only if $C^*_{\ell^p}(\Gamma)=C^*(\Gamma)$ for every $1\leq p<\infty$.

Suppose $G$ is a locally compact group, $H$ is a closed subgroup, and $\mu$ a quasi-invariant measure on the homogeneous space $G/H$. We say the homogeneous space $G/H$ is amenable if $L^\infty(G/H,\mu)$ admits a $G$-invariant mean (see \cite{e}). This leads one to consider the question: can we give an analagous characterization of amenability of $\Gamma/H$ as mentioned above? For fixed $H\leq \Gamma$, define
$$ D_p=D_p(H)=\{f\in \ell^\infty(\Gamma) : f|_{sHt}\in \ell^p(sHt)\tn{ for }s,t\in \Gamma\}. $$
Is it the case that $C_{D_p}^*(\Gamma)=C^*_{\ell^p}(\Gamma)$ if and only if $\Gamma/H$ is amenable? In this case, taking $H$ to be the trivial subgroup would recover the original result. Unfortunately, we do not know the answer to this question but we have attained some partial results including the reverse implication.

\begin{prop}\label{2}
Suppose $\Gamma/H$ is amenable. Then $C_{D_p}^*(\Gamma)=C_{\ell^p}^*(\Gamma)$ for every $p\in [1,\infty)$.
\end{prop}

\begin{proof}

For $s\in \Gamma$ and $f\fn \Gamma/H\to \C$, we let $f_s$ denote the left translation of $f$ by $s$. Since $\Gamma/H$ is amenable, there exists a net $\{f_i\}$ of finitely supported functions in $\Gamma/H$ with $\|f_i\|_2=1$ so that $\|(f_i)_{s}-f_i\|_2\to 0$ for every $s\in \Gamma$ \cite[p. 28]{e}. Let $\sigma$ be the induced representation $\mathrm{Ind}_H^\Gamma 1_H$. Then $\sigma_{f_i,f_i}$ are positive definite functions converging pointwise to the trivial representation.

Fix a positive definite function $\varphi\in D_p$. Then, since $\sigma_{f_i,f_i}$ is supported on only finitely many cosets $sH$ for each $i$, we have that $\varphi\sigma_{f_i,f_i}\in \ell^p$ and, hence, extends to a positive linear functional on $C^*_{\ell^p}(\Gamma)$ for every $i$. Since it is also the case that $\varphi\sigma_{f_i,f_i}\to \varphi$ pointwise, we conclude that $\varphi$ extends to a positive linear functional on $C_{\ell^p}^*(\Gamma)$.

Now let $\varphi$ be an arbitrary positive linear functional on $C^*_{D_p}(\Gamma)$. Then we can find a net $\{\varphi_i\}$ of sums of positive definite functions associated to $D_p$-representations converging pointwise to $\varphi$. By approximating each $\varphi_i$ by positive definite functions in $D_p$, we may assume that $\{\varphi_i\}\subset D_p$. Then, since each $\varphi_i$ extends to a positive linear functional on $C^*_{\ell^p}(\Gamma)$ and $\varphi$ is the pointwise limit of these positive definite functions, we conclude that $\varphi$ extends to a positive linear functional on $C^*_{\ell^p}(\Gamma)$. Hence, $\|x\|_{D_p}\leq \|x\|_{\ell^p}$ for every $x\in \C[\Gamma]$. As the reverse inequality is clear, we conclude that $C_{D_p}^*(\Gamma)=C^*_{\ell^p}(\Gamma)$.
\end{proof}

We are yet to determine whether the forward implication is also true, but we are able to show it in the modest case when $\Gamma$ is the direct product $H\times K$.

\begin{prop}\label{1}
Let $p\in [1,\infty)$ and suppose $\Gamma=H\times K$. If $C^*_{\ell^p}(\Gamma)=C^*_{D_p(H)}(\Gamma)$, then $K$ is amenable. 
\end{prop}

\begin{proof}

Suppose that $C^*_{\ell^p}(\Gamma)=C^*_{D_p}(\Gamma)$ and let $\omega\in\ell^p(H)$ be a normalised positive definite function on $H$. Define $\varphi\fn \Gamma\to \C$ by $\varphi(h,k)=\omega(h)$. Then $\varphi$ is a positive definite function which lies in $D_p(H)$ and, hence, extends to a positive linear functional on $C^*_{\ell^p}(\Gamma)$. So we can find a net $\{\varphi_i\}$ of positive definite functions in $\ell^p$ converging pointwise to $\varphi$.

Choose $n$ large enough so that $p/n\leq 2$ and define $\psi=\varphi^n$, $\psi_i=\varphi_i^n$. Then $\{\psi_i\}$ is a net of $\ell^2$-summable positive definite functions convergiving pointwise to $\psi$. Hence, $\psi$ extends to a positive linear functional on $C_r^*(\Gamma)$. Thus, there is a net $\{f_i\}\subset \ell^2(\Gamma)$ with $\|f_i\|=1$ so that $\{\lambda_{f_i,f_i}\}$ converges to $\psi$ pointwise (sums of positive definite functions associated to $\lambda$ can be written in this form by \cite[p. 218]{fa}).

Define $g_i\fn K\to \C$ by $g_i(k)=\|f_i|_{H\times \{k\}}\|_2$. Then $\|g_i\|_2=1$ for every $i$. Further,
\begin{eqnarray*}
\big|\lambda_{f_i,f_i}(e,k)\big| &=& \bigg|\sum_{(h,k')\in H\times K} f_i(h,k^{-1}k')f_i(h,k')\bigg| \\
&\leq & \sum_{k'\in K}\big\|(f_i)_{(e,k^{-1})}|_{H\times \{k'\}}\cdot f_i|_{H\times\{k'\}}\big\|_1 \\
&\leq & \sum_{k'\in K} \big\|f_i|_{H\times \{k^{-1}k'\}}\big\|_2\big\|f_i|_{H\times\{k'\}}\big\|_2 \\
&=& \sum_{k'\in K} g_i(k^{-1}k')g_i(k') \\
&=& \lambda_{g_i,g_i}(k) \leq 1.
\end{eqnarray*}
Consequently, $\{\lambda_{g_i,g_i}\}$ converges pointwise to the trivial representation since $\{\lambda_{f_i,f_i}(e,k)\}$ converges to $\psi(e,k)=1$ for every $k\in K$. Hence, $K$ is amenable.
\end{proof}

Brown and Guentner demonstrated in \cite[Proposition 2.11]{bg} that $C_{\ell^p}^*(\Gamma)=C_r^*(\Gamma)$ for every $p\in [1,2]$. It is natural to wonder if this continues to hold true for $p\in (2,\infty)$, however this is not the case. Let $\mathbb{F}_d$ be the free group on $d$ generators for fixed $2\leq d<\infty$. Brown and Guentner were able to show that there exists $p\in (2,\infty)$ so that $C^*(\mathbb F_d)\neq C^*_{\ell^p}(\mathbb F_d)\neq C^*_r(\Gamma)$ \cite[Proposition 4.4]{bg}. Subsequently, Higson, Ozawa, and Okayasu \cite[Corollary 3.7]{o} independently showed that the $C^*_{\ell^p}(\mathbb F_d)$ are all distinct $C^*$-algebras for $p\in (2,\infty)$. This allows us to conclude that if $\Gamma$ contains a copy of the free group, then $C^*_{\ell^p}(\Gamma)$ are distinct for $2\leq p<\infty$:

\begin{remark}\label{3}
If $H$ is any subgroup of $\Gamma$ and $\psi$ an $\ell^p$-summable positive definite function on $H$, then we may naively extend $\psi$ to an $\ell^p$-summable positive definite function $\varphi$ on $\Gamma$ by defining $\varphi(s)=\psi(s)$ when $s\in H$ and $\varphi(s)=0$ otherwise. It follows that $C^*_{\ell^p}(\Gamma)$ are all distinct for $p\in [2,\infty)$ and $\Gamma$ containing a copy of the free group.
\end{remark}

This remark leads to the question: what other extension type results exist? The following theorem shows that the induced representation of an $\ell^p$-representation remains an $\ell^p$-representation.

\begin{theorem}\label{induced}
Let $H$ be a subgroup of the discrete group $\Gamma$ and $\sigma\fn H\to\Hi$ an $\ell^p$-representation of $H$. Then $\pi:=\mathrm{Ind}_H^\Gamma\sigma$ is an $\ell^p$-representation of $\Gamma$.
\end{theorem}

\begin{proof}
Let $q\fn \Gamma\to \Gamma/H$ denote the canonical quotient map. Recall that the induced representation $\pi$ is given by left translation on the completion $\mc F$ of the space
$$ \mc F_0=\{f\fn \Gamma\to \Hi \mid q(\mathrm{supp}\,f)\tn{ is finite and } f(s\xi)=\sigma(\xi^{-1})f(s) \tn{ for all }s\in\Gamma, \xi\in H\} $$
with respect to the inner product
$$ \lla f,g\rra=\sum_{tH\in\Gamma/H}\lla f(t),g(t)\rra_\sigma.$$

Let $\Hi_0$ be a dense linear subspace of $\Hi$ such that $\sigma_{x,y}\in\ell^p(H)$ for every $x,y\in\Hi_0$ (if $\pi_{x,x}\in \ell^p(H)$ for every $x\in \Hi_0$, then $\pi_{x,y}\in \ell^p(H)$ for all $x,y\in\Hi_0$ by the polarization identity). Fix a set of representatives $\{r_i\}_{i\in \Gamma/H}$ for $\Gamma/H$. Then the span of the functions $f\in \mc F_0$ such that $f(r_i)$ is nonzero for at most one $i$ and $f(r_i)\in\Hi_0$ is dense in $\mc F$.

Fix $f$ and $g$ as above. Without loss of gerenality, we may assume that $f$ and $g$ are nonzero. Let $i$ and $j$ be the indices such that $f(r_i)\neq 0$, $g(r_j)\neq 0$. Then
\begin{eqnarray*}
\sum_{s\in \Gamma} |\pi_{f,g}(s)|^p &=& \sum_{s\in\Gamma} \bigg| \sum_{k\in \Gamma/H} \lla f(s^{-1}r_k), g(r_k)\rra \bigg |^p \\
&=& \sum_{s\in\Gamma} \big| \lla f(s^{-1}r_j),g(r_j)\rra \big|^p \\
&=& \sum_{\xi\in H} \big| \lla f(r_i\xi),g(r_j)\rra \big|^p \\
&=& \sum_{\xi\in H} \big| \lla \sigma(\xi^{-1})f(r_i),g(r_j)\rra\big|^p \\
&=& \|\sigma_{f(r_i),g(r_j)}\|_p^p<\infty.
\end{eqnarray*}
It follows that $\pi$ is an $\ell^p$-representation.
\end{proof}

Much of the attention in this section has been focused towards the chain of $\ell^p$ ideal completions. This raises the question, can we find exotic group $C$*-algebras which lie off this chain? We end this section by showing that $D_p$ ideal completions can satisfy this criteria.

\begin{example}
Fix $p\in [2,\infty)$ and let $\F_\infty$ be the free group on countably many generators $a_1,a_2,\ldots$ and view $\F_d$ as the subgroup of $\F_\infty$ generated by $a_1,\ldots,a_d$. Take $H=\mathbb F_d$ for some fixed $d\geq 2$. Let $\varphi_\alpha\fn \F_\infty\to \C$ be the positive definite function defined by $\varphi_\alpha(s)=\alpha^{|s|}$ for each $\alpha\in (0,1)$ (see \cite[Lemma 1.2]{h}). Then $\varphi_\alpha\in D_p=D_p(\F_d)$ for each $\alpha<(2d-1)^{-1/p}$ since
$$\sum_{s\in \F_d}\alpha^{-|t_1|}\alpha^{-|t_2|}\varphi_\alpha(s)\leq \sum_{s\in \F_d}\varphi_\alpha(t_1st_2)\leq \sum_{s\in \F_d}\alpha^{|t_1|}\alpha^{|t_2|}\varphi_\alpha(s)$$
for every $t_1,t_2\in \F_\infty$ and $\sum_{s\in \F_d}\varphi_\alpha(s)<\infty$ if and only if $\alpha<(2d-1)^{-1/p}$. Hence, $\varphi_\alpha$ extends to a positive linear functional on $C^*_{D_p}(\F_\infty)$ for each $\alpha\leq (2d-1)^{-1/p}$. By \cite[Corollary 3.5]{o}, we have that $\varphi_\alpha|_{\F_d}$ extends to a positive linear functional on $C^*_{\ell^p}(\F_d)$ if and only if $\alpha\leq (2d-1)^{-1/p}$. Therefore, this condition of $\alpha\leq (2d-1)^{-1/p}$ is necessary and sufficient for $\varphi_\alpha$ to extend to a positive linear functional on $C^*_{D_p}(\F_\infty)$.

Fix $\alpha\in (0,1)$ and choose a positive integer $d'$ large enough so that $(2d'-1)^{-1/p}<\alpha$. Then $\varphi_\alpha|_{\F_{d'}}$ does not extend to a positive linear functional on $C^*_{\ell^p}(\F_{d'})$. Hence, $\varphi_\alpha$ does not extend to a positive linear functional on $C^*_{\ell^p}(\F_\infty)$ for any $\alpha\in (0,1)$. Therefore there is no canoncial quotient map from $C^*_{\ell^q}(\F_\infty)$ to $C^*_{D_p}(\F_\infty)$ for any $p,q\in [2,\infty)$. Conversely, by Remark \ref{3}, there is no canonical quotient map from $C^*_{D_p}(\F_\infty)$ to $C^*_{\ell^q}(\F_\infty)$ for any $q>p\geq 2$.

Now suppose $\Gamma$ is an arbitrary group containing a copy of the free group. Since the free group on two generators contains an isomorphic copy of $\mathbb F_\infty$, a similar argument as above can be used to justify there is a subgroup $H\leq \Gamma$ so that $C^*_{D_p}(\Gamma)$ lie off the chain $C^*_{\ell^p}(\Gamma)$.

\end{example}

\section{The lattice of group $C$*-algebras}

So far we have nonchalantly been talking about notions such as chains of group $C$*-algebras. We are able to do this because there is a natural partial ordering which can be placed on the group $C$*-algebras. In this section, we make this notion a partial ordering specific and show that the group $C$*-algebras form a complete $\bigvee$-semilattice. This will allow us to build new exotic group $C$*-algebras.

\begin{defn}
Let $\Gamma$ be a discrete group. By a group $C$*-algebra (associated to $\Gamma$), we will mean a $C$*-completion of $\C[\Gamma]$. Place a partial ordering on the group $C$*-algebras by saying that $\mc A\preceq\mc A'$ if $\|x\|_{\mc A}\leq \|x\|_{\mc A'}$ for every $x\in \C[\Gamma]$. Equivalently, we have that $\mc A\preceq \mc A'$ if and only if the identity map on $\C[\Gamma]$ extends to a quotient from $\mc A'$ to $\mc A$.
\end{defn}

We observe that with this definition, the group $C$*-algebras form a complete $\bigvee$-semilattice. Indeed, if $\{\mc A_i\}$ is a collection of group $C$*-algebras, then the completion of $\C[\Gamma]$ with respect to the $C$*-norm $\|\cdot\|$ defined by $\|x\|=\sup_i\|x\|_{\mc A_i}$ for $x\in \C[\Gamma]$ is the join $\bigvee_i \mc A_i$.

Note that it also makes sense to talk about the supremum and infimum of ``completions'' of $\C[\Gamma􀀀]$ with respect to a $C$*-seminorm. If $\mc S$ is a collection of representations of $\Gamma$, then $C^*_{\mc S}(\Gamma)$ is defined to be the ``completion'' of $\C[\Gamma]$ with respect to the $C$*-seminorm $\|x\|_{\mc S}:=\sup_{\pi\in \mc S}\|\pi(x)\|$ (as in \cite{e}). Moreover, every $C$*-seminorm arises in this way where $\mc S$ can be assumed to be a Fell closed subset of the irreducible representations $\widehat{\Gamma}$ \cite[Proposition F.2.7]{bpt}. Further, if $\mc S$ and $\mc S'$ are Fell closed subsets of $\widehat \Gamma$, then $C^*_{\mc S}(\Gamma)\preceq C^*_{\mc S'}(\Gamma)$ if and only if $\mc S\subset \mc S'$. If we place the same lattice structure on these ``completions'' as above, we get a complete lattice. Indeed, let $\mc\{\mc A_i\}$ be a collection of such ``completions'' and write $\mc A_i=C^*_{\mc S_i}(\Gamma)$ for Fell closed subsets $\{\mc S_i\}\subset\widehat{\Gamma}$. Then $\bigwedge_i \mc A_i=C^*_{\cap \mc S_i}$. As the join arises as before, we conclude that we indeed get a complete lattice.



In the remainder of this section, we focus on using this lattice structure to produce a new class of examples of exotic group $C$*-algebras. Towards this goal, we give a characterization of when $C^*_D(\Gamma)= C^*(\Gamma)$ in terms of when the $D$-representations weakly contain an amenable representation, i.e., a representation $\pi\fn \Gamma\to B(\Hi)$ for which there exists a state $\mu$ on $B(\Hi)$ such that $\mu(\pi(s)T\pi(s^{-1})=\mu(T)$ for all $s\in\Gamma$ and $T\in B(\Hi)$ (see \cite{b-am}).

\begin{prop}\label{amenable}
Let $D\triangleleft \ell^\infty(\Gamma)$ be an algebraic ideal. Then $C^*_D(\Gamma)=C^*(\Gamma)$ if and only if the $D$-representations weakly contains an amenable representation.
\end{prop}

\begin{proof}
If $C^*_D(\Gamma)=C^*(\Gamma)$, then the $D$-representations weakly contain all representations of $\Gamma$ and, in particular, contains the trivial representation which is evidently amenable.

Conversely, suppose that $D$ weakly contains an amenable representation $\pi$. Then we can find a net $\{\pi_i\}$ of $D$-representations converging in Fell's topology to $\pi$. Note that $\overline{\pi}\otimes \pi_i$ is a $D$-representation for every $i$ and $\overline{\pi}\otimes\pi_i$ converges to $\overline{\pi}\otimes\pi$ in the Fell topology. So the $D$-representations weakly contain $\overline{\pi}\otimes\pi$ and, hence, the trivial representation \cite[Theorem 5.1]{b-am}. A similar argument as used above now shows that the $D$-representations weakly contain every representation of $\Gamma$ since $\pi\otimes 1=\pi$ for every $\pi$. Hence, $C^*_D(\Gamma)=C^*(\Gamma)$.
\end{proof}

In particular, this proposition shows that if $C^*_D(\Gamma)\neq C^*(\Gamma)$, then the $D$-representations do not weakly contain any finite dimensional representations. Hence, if $C^*_D(\Gamma)$ is a group $C$*-algebra not coinciding with $C^*(\Gamma)$, then $C^*_D(\Gamma)\vee C^*_\pi(\Gamma)$ is a strictly larger $C$*-algebra than $C^*_D(\Gamma)$ for every finite dimensional representation and the only group $C$*-algebra $\A$ produced by an ideal completion for which $\A\succeq C^*_D(\Gamma)$ is $C^*(\Gamma)$.

Suppose $\Gamma$ contains a copy of the free group and $\pi$ is a finite dimensional representation of $\Gamma$. What does the group $C$*-algebra $C^*_D(\Gamma)\vee C^*_\pi(\Gamma)$ look like when we take $D=\ell^p$? Could it be the case that $C^*_{\ell^p}(\Gamma)\vee C^*_\pi(\Gamma)$ coincides with $C^*(\Gamma)$? Could $C^*_{\ell^p}(\Gamma)\vee C^*_\pi(\Gamma)$ dominate $C^*_{\ell^q}(\Gamma)$ for some $q>p\geq 2$? It turns out that neither of these cases can occur:

\begin{prop}
Suppose $\Gamma$ contains a copy of the free group and $\mc F_0$ is a finite nonempty subset of the finite dimensional representations on $\Gamma$. Then $C^*_{\ell^q}(\Gamma)\not\preceq C^*_{\mc F_0}(\Gamma)\vee C^*_{\ell^p}(\Gamma)$ for any $q>p\geq 2$.
\end{prop}

\begin{proof}
Without loss of generality, we may assume that $\mc F_0$ is a subset of $\widehat{\Gamma}$. For each $p\geq q$, write $C^*_{\ell^p}(\Gamma)=C^*_{\mc S_p}(\Gamma)$ for some Fell closed subset $\mc S_p\subset \widehat{\Gamma}$. Then, since $\mc S_p$ is a proper subset of $\mc S_q$ for every $q>p$, $\mc S_q\backslash \mc S_p$ has inifinite cardinality for $q>p$ (as there is an infinitude of intermediate $C$*-algebras between $C^*_{\ell^p}(\Gamma)$ and $C^*_{\ell^p}(\Gamma)$). Now, write $C^*_{\ell^p}(\Gamma)\vee C^*_{\mc F_0}(\Gamma)=C^*_{\mc S_0}(\Gamma)$ for some Fell closed $S_0\subset \widehat{\Gamma}$. Then since, $\mc F_0$ is a closed subset of $\widehat{\Gamma}$ in the Fell topology, $\mc S_0\backslash \mc S=\mc F_0$ has finite cardinality. Hence, $C^*_{\ell^p}(\Gamma)\vee C^*_{\mc F_0}(\Gamma)\not\succeq C^*_{\ell^q}(\Gamma)$ for $q>p$.
\end{proof}

This gives us another class of examples of exotic group $C$*-algebras which lie off the chain $C^*_{\ell^p}(\Gamma)$. Note that we may always take $\mc F_0$ to be the singleton containing the trivial representation. Hence, for $\Gamma$ containing a copy of the free group, this construction can always be used to produce exotic group $C$*-algeras differing from $C^*_{\ell^p}(\Gamma)$.

\section{Exotic group $C$*-algebras of $SL_n(S)$}

Let $G$ be a locally compact group and $G_d$ be the group $G$ endowed with the discrete topology. Denote the left regular representation of $G$ by $\lambda_G$. Then $\lambda_G$ is a representation of $G_d$. In \cite{bkls} Bekka, Kaniuth, Lau, and Schlitchting show that the group $C$*-algebra $C^*_{\lambda_G}(G_d)$ is the reduced $C$*-algebra if and only if $G$ admits an open subgroup $H$ so that $H_d$ is amenable. In particular, if $G$ is a connected nonamenable group such as $SL_n(\R)$ ($n\geq 2$), then $C^*_{\lambda_G}(G_d)\neq C^*_r(G_d)$. This inspires us to study the group $C$*-algebra $C^*_\delta(SL_n(S)):=C^*_{\lambda_{SL_n(\R)}}(SL_n(S))$ in this section when $S$ is taken to be a dense (unital) subring of $\R$. We will demonstrate that $C^*_\delta(SL_n(S))$ is an exotic group $C$*-algebra and compare it to the $C$*-algebras $C^*_{\ell^p}(SL_n(S))$. Our first proposition shows that $C^*_{\ell^p}(SL_n(S))\not\succeq C^*_{\delta}(SL_n(S))$ for any $1\leq p<\infty$.

Before proceeding to this proposition, we mention a result due to Breuillard and Gellander which we make use of. In \cite{freeLie}, these authors demonstrated that if $\Gamma$ is a dense subgroup of a connected semi-simple real Lie group $G$, then $\Gamma$ contains a copy of the free group on two generators which is dense in $G$. Moreover, their proof shows that these generators can be chosen arbitrarily close to the identity.

\begin{prop}\label{dense Lie}
Let $\Gamma$ be a dense subgroup of a connected semi-simple real Lie group $G$. If a representation $\pi$ of the discrete group $\Gamma$ is continuous in the ambient topology, then $\pi$ is not weakly contained in the $\ell^p$-representations for each $1\leq p<\infty$.
\end{prop}

(Compare this statement to that of \cite[Proposition 5]{bkls}.)

\begin{proof}
Let $\pi$ be a continuous representation of $\Gamma$ and $x\in\Hi_\pi$ be a unit vector. Then $\varphi:=\pi_{x,x}$ is a continuous function on $\Gamma$ with $\varphi(e)=1$. We will demonstrate that $\varphi$ does not extend to a state on $C^*_{\ell^p}(\Gamma)$ for any $1\leq p<\infty$.

In the proof of \cite[Theorem 3.4 (2)]{o}, Okayasu shows that a normalized positive definite function $\psi$ on $\F_2$ extends to a state on $C^*_{\ell^p}(\mathbb F_2)$ if and only if $\|\psi\chi_k\|_p\leq k+1$ (where $\chi_k$ is the characteristic function of the words of length $k$). Choose $k$ large enough so that $(4\cdot 3^{k-1})^{1/p}>k+1$. Next choose generators for a free subgroup of $\Gamma$ close enough to the identity so that $|\psi
(s)|>\frac{k+1}{(4\cdot3^{k-1})^{1/p}}$ for all $s\in W_k$. Then, on this copy of $\F_2$,
$$ \|\psi\chi_k\|_p>\left( |W_k|\left(\frac{k+1}{(4\cdot 3^{k-1})^{1/p}}\right)^p\right)^{1/p}=k+1. $$
Hence, $\psi$ does not extend to a positive linear functional on $C^*_{\ell^p}(\mathbb F_2)$ and, so, we conclude that $\pi$ is not weakly contained in the $\ell^p$-representations.
\end{proof}

To observe that the following proposition applies to our situation, we note that $SL_n(S)$ is a dense subgroup of $SL_n(\R)$ since the subgroups
$$\begin{bmatrix}
1 & & \multicolumn{1}{c}{\multirow{2}{*}{\huge *}} \\
\multicolumn{1}{c}{\multirow{2}{*}{\vspace{-5pt}\LARGE 0}} & \ddots & \\
& & 1
\end{bmatrix}
\ \mathrm{ and }\ 
\begin{bmatrix}
1 & & \multicolumn{1}{c}{\multirow{2}{*}{\vspace*{5pt}\LARGE 0}} \\
\multicolumn{1}{c}{\multirow{2}{*}{\vspace{-15pt}\huge *}} & \ddots & \\
& & 1
\end{bmatrix}
$$
generate $SL_n(\R)$.

This proposition demonstrates that $C^*_\delta(SL_n(S))$ is a strictly larger group $C$*-algebra than $C^*_r(SL_n(S))$. Bekka, Kaniuth, Lau and Schlitchting's result \cite[Proposition 5]{bkls} implies that this is the case when $S$ is taken to be all of $\R$, but it was not apriori obvious that this would continue to hold for smaller rings $S$.

We are now led to ask similar questions as in the previous section. Could $C^*_\delta(SL_n(S))$ be the full group $C^*(SL_n(S))$? How does $C^*_\delta(SL_n(S))\vee C^*_{\ell^p}(SL_n(S))$ compare to $C^*_{\ell^q}(SL_n(S)$? These questions are quite satisfactorily answered in the following proposition.

\begin{prop}
Suppose $q>p\geq 2$. Then $C^*_\delta(SL_n(S))\vee C^*_{\ell^p}(SL_n(S))\not\succeq C^*_{\ell^q}(SL_n(S))$.
\end{prop}

\begin{proof}
Suppose that $\varphi$ is a normalized positive definite function on $SL_n(S)$ which extends to a state on $C^*_\delta(SL_n(S))$. We will show that $\varphi|_{SL_n(\Z)}$ extends to a state on $C^*_r(SL_n(\Z))$.

Since $\varphi$ extends to a state on $C^*_\delta(SL_n(S))$. Then we can find a net $\{\varphi_i\}$ of sums of positive definite functions associated to $\lambda_{SL_n(\R)}$ which converge pointwise to $\varphi$. By Herz's restriction theorem \cite{herz}, $\varphi_i|_{SL_n(Z)}$ lies in $A(SL_n(\Z))$, the Fourier algebra of $SL_n(\Z)$ (see \cite{fa} for a reference on the Fourier algebra). Hence, as each $\varphi_i|_{SL_n(\Z)}$ is positive definite and $\varphi_i|_{SL_n(\Z)}$ converges pointwise to $\varphi|_{SL_n(\Z)}$, we conclude that $\varphi|_{SL_n(\Z)}$ extends to a positive linear functional on $C^*_r(SL_n(\Z))$.

This shows us that $\|x\|_\delta\leq \|x\|_{\ell^2}$ for every $x\in \C[SL_n(\Z)]$ since $\|x\|_\delta^2=\|x^*x\|_\delta=\sup\varphi(x^*x)$ where the supremum is taken over states $\varphi$ on $C^*_\delta(SL_n(S))$. Hence
$$\|x\|_{C^*_\delta(SL_n(S))\vee C^*_{\ell^p}(SL_n(S))}=\|x\|_{\ell^p}$$
for $x\in \C[SL_n(\Z)]$. As $\|\cdot\|_{\ell^q}$ is a larger norm on $\C[SL_n(\Z)]$ than $\|\cdot\|_{\ell^p}$, we conclude that $C^*_\delta(SL_n(S))\vee C^*_{\ell^p}(SL_n(S))\not\succeq C^*_{\ell^q}(SL_n(S))$.
\end{proof}

We note that this proposition adds to our list of examples of exotic group $C$*-algebras.

\begin{remark}
Suppose $G$ is a nonamenable group containing a discrete copy of the free group. A similar analysis shows that $C^*_{\lambda_G}(G_d)$ is not the full group $C$*-algebra. This justifies our comment in the introduction that $C^*_{\lambda_G}(G_d)$ leads to a large class of exotic group $C$*-algebras.
\end{remark}

\section{Exotic group $C$*-algebras of $SL_n(\Z)$}

Let $\mc F$ denote the set of finite dimensional representations on $SL_n(\Z)$. Notice that the natural homomorphisms from $SL_n(\Z)$ to $SL_n(\Z/N\Z)$ for $N\geq 1$ separate the points of $SL_n(\Z)$. Hence,  $C^*_{\mc F}(SL_n(\Z))$ is a group $C$*-algebra \cite[Proposition 1]{b}. Since $SL_n(\Z)$ is nonamenable, we have by Proposition \ref{amenable} that the left regular representation does not weakly contain any finite dimensional representations. Hence, $C^*_{\mc F}(SL_n(\Z))$ is strictly larger than the reduced $C$*-algebra.

In \cite{b}, Bekka demonstrates that the universal $C$*-algebra $C^*(SL_n(\Z))$ is not residually finite dimensional for $n\geq 3$, hence showing that $C^*_{\mc F}(SL_n(\Z))$ is an exotic $C$*-algebra for $n\geq 3$. Let $\mc F_0$ denote the set of finite dimensional representations which factor through a congruence subgroup of $SL_n(\Z)$ (this is the kernel $\Gamma(N)$ of the natural map from $SL_n(\Z)$ to $SL_n(\Z/N\Z)$). What Bekka actually showed was that $C^*_{\mc F_0}(SL_n(\Z))$ is not the full group $C$*-algebra for $n\geq 2$ and that $\mc F_0=\mc F$ for $n\geq 3$.

%

Our questions about the exotic group $C$*-algebra $C^*_{\mc F_0}(SL_n(\Z))$ are again similar to those asked in the previous two sections. Our first is how does $C^*_{\mc F_0}(SL_n(\Z))$ compare to $C^*_{\ell^p}(SL_n(\Z))$? We provide a partial answer to this question below.

\begin{prop}
Let $n\geq 2$. There exists a $p\in (2,\infty)$ so that $C^*_{\mc F_0}(SL_n(\Z))\not\succeq C^*_{\ell^p}(SL_n(\Z))$.
\end{prop}

\begin{proof}
Note that if $\pi$ is a representation of $SL_n(\Z)$ which factors through a congruence subgroup, then the restriction of $\pi$ to $SL_2(\Z)$ factors through a congruence subgroup of $SL_2(\Z)$. Hence, it suffices to consider the case when $n=2$.

Note that for each $\alpha\in (0,1)$, the positive definite function $\varphi_\alpha\fn \F_2\to \C$ defined by $\phi_\alpha(s)=\alpha^{-|s|}$ lies in $\ell^p(\F_2)$ for some $p$. Let $\pi_\alpha$ denote the GNS representation of $\varphi_\alpha$. Then, since $\varphi_\alpha$ converges pointwise to the trivial representation as $\alpha\nearrow 1$, $\pi_\alpha$ converges to $1_{\F_2}$ in the Fell topology.

Let $\F_2\subset SL_2(\Z)$ be a finite index embedding of the free group in $SL_2(\Z)$ \cite{serre}. Then $\mathrm{Ind}_{\mathbb F_2}^{SL_2(\Z)}\pi_\alpha$ converges to $\mathrm{Ind}_{\F_2}^{SL_2(\Z)}1_{\F_2}$. Note that $\mathrm{Ind}_{\F_2}^{SL_2(\Z)}1_{\F_2}$ contains a copy of $1_{SL_2(\Z)}$ as a subrepresentation since $\F_2$ is of finite index in $SL_2(\Z)$. Hence, $\mathrm{Ind}_{\mathbb F_2}^{SL_2(\Z)}\pi_\alpha\to 1_{SL_2(\Z)}$ in the Fell topology.

Fix a finite generating set $S$ for $SL_2(\Z)$. In \cite[Lemma 3]{b} Bekka shows that the trivial representation is isolated among the set of restrictions $\pi|_{SL_2(\Z)}$ where $\pi$ is a representation which factors through a congruence subgroup. This is to say that there exists $\epsilon>0$ so that if $\pi\fn SL_2(\Z)\to B(\Hi)$ is a representation which factors through a congruence subgroup with the property that there exists $x\in \Hi$ so that $\|\pi(s)x-x\|<\epsilon$ for every $s\in S$, then $\pi$ contains the trivial representation as a subrepresentation.

Since $\mathrm{Ind}_{\F_2}^{SL_2(\Z)}\pi_\alpha\to 1_{SL_2(\Z)}$ in the Fell topology, we can find $\alpha$ and a unit vector $x$ in the corresponding Hilbert space so that $\|\pi_\alpha(s)x-x\|<\epsilon$. By Theorem \ref{induced}, $\mathrm{Ind}_{\F_2}^{SL_2(\R)}\pi_\alpha$ is an $\ell^p$-representation for some $p$ and, hence, does not weakly contain a copy of the trivial representation. Therefore $\pi_\alpha$ is not weakly contained in $\mc F_0$.
\end{proof}

Note that since $C^*_{\ell^q}(SL_n(\Z))\succeq C^*_{\ell^p}(SL_n(\Z))$ when $q>p$, the proposition provides the same conclusion for all $q>p$. This answer to our question is not as clean as that provided in the previous section and we are still left with questions. Does the conclusion of the proposition hold for any $p>2$? If not, can we provide nontrivial estimates on the values of $p$ which provide the conclusion of the proposition?

We conclude this paper by showing that $C^*_{\mc F_0}(\Gamma)\vee C^*_{\ell^p}(\Gamma)$ forms a class of exotic group $C$*-algebras. This proposition provides a similar conclusion as the last:

\begin{proposition}
Let $n\geq 2$. For every $p\in [1,\infty)$, there exists $q>p$ so that $C^*_{\mc F_0}(SL_n(\Z))\vee C^*_{\ell^p}(SL_n(\Z))\not\succeq C^*_{\ell^q}(\Gamma)$.
\end{proposition}

\begin{proof}
Again it suffices to consider the case when $n=2$.

Let $S$ and $\epsilon$ be as in the previous proposition. Take $\pi_p\fn SL_n(\Z)\to B(\Hi_p)$ to be a faithful representation of $C^*_{\ell^p}(SL_2(\Z))$. Then, since the $\ell^p$-representations do not weakly contain the trivial representation, there exists $\epsilon'>0$ so that whenever $x\in \Hi_p$ is a unit vector, there exists $s\in S$ so that $\|\pi_p(s)x-x\|\geq \epsilon'$.

Suppose that $\sigma\fn SL_n(\Z)\to B(\Hi_\sigma)$ in $\mc F_0$ does not contain a copy of the trivial representation. Take $(x,y)\in\Hi_p\oplus \Hi_\sigma$ to be a unit vector. If $\|x\|\geq 1/\sqrt{2}$, there exists $s\in S$ so that $\|\pi_p(s)x-x\|\geq\epsilon'/\sqrt 2$ which implies that $\|(\pi_p\oplus \sigma)(s)(x,y)-(x,y)\|\geq \epsilon'/\sqrt{2}$. Similarly, if $\|y\|\geq 1/\sqrt{2}$, $\|(\pi_p\otimes \sigma)(s)(x,y)-(x,y)\|\geq \epsilon/\sqrt{2}$. Hence, $\|(\pi_p\oplus \sigma)(s)(x,y)-(x,y)\|\geq \min\{\epsilon,\epsilon'\}/\sqrt 2$.

Since we have that $\pi_q\to 1$ in the Fell topology as $q\to\infty$, we can find $q$ so that there exists a unit vector $x\in \Hi_q$ with $\|\pi_q(s)x-x\|<\min\{\epsilon,\epsilon'\}/\sqrt 2$ for every $s\in S$. As $\pi_q$ does not weakly contain the trivial representation, we conclude that $\pi_q$ is not weakly contained in $\{\pi_p\}\cup \mc F_0$.
\end{proof}

\section*{Acknowledgements}
The author would like to thank his advisor Nico Spronk for suggesting this project. The author also wishes to thank both his advisors, Brian Forrest and Nico Spronk, for many useful discussions and suggestions. The author was supported by an NSERC Postgraduate Scholarship.

\bibliographystyle{amsplain}
\bibliography{exotic}

\end{document}